\documentclass[12pt,a4paper]{amsart}  
\usepackage{amsmath,amsthm,amssymb,amsfonts,graphicx}
\usepackage{underscore}
\usepackage{cmbright}
\usepackage{mathtools}
\usepackage[T1]{fontenc}
\usepackage[utf8]{inputenc}
\usepackage{stmaryrd,float}
\usepackage{enumerate,paralist}
\usepackage[usenames,dvipsnames,svgnames,table]{xcolor}

\usepackage[bmargin=30mm,tmargin=30mm,lmargin=30mm,rmargin=30mm]{geometry}
\usepackage[
pdftitle = {Tree Densities in Sparse Graph Classes},
pdfauthor = {Huynh -- Wood},
colorlinks =true,
linkcolor={black},
urlcolor={blue!80!black},
filecolor={blue!80!black},
citecolor={black},
anchorcolor=blue,
filecolor=blue, 
menucolor=blue]{hyperref} 

\usepackage[noabbrev,capitalise]{cleveref}
\crefname{lem}{Lemma}{Lemmas}
\crefname{thm}{Theorem}{Theorems}
\crefname{cor}{Corollary}{Corollaries}
\crefname{prop}{Proposition}{Propositions}
\crefname{conj}{Conjecture}{Conjectures}
\crefname{open}{Open Problem}{Open Problems}
\crefname{obs}{Observation}{Observations}

\crefformat{equation}{(#2#1#3)}
\Crefformat{equation}{Equation #2(#1)#3}

\theoremstyle{plain}
\newtheorem{thm}{Theorem}
\newtheorem{lem}[thm]{Lemma}
\newtheorem{cor}[thm]{Corollary}

\newtheorem{conj}[thm]{Conjecture}

\usepackage[numbers,sort&compress]{natbib}
\usepackage{hypernat}
\bibpunct{\nolinebreak{}[}{]}{,}{n}{}{,}

\usepackage{etoolbox}
\patchcmd{\section}{\scshape}{\bfseries}{}{}
\makeatletter
\renewcommand{\@secnumfont}{\bfseries}
\makeatother

\DeclarePairedDelimiter\ceil\lceil\rceil
\DeclarePairedDelimiter\floor\lfloor\rfloor

\renewcommand{\geq}{\geqslant}
\renewcommand{\leq}{\leqslant}

\newcommand{\RR}{\mathbb{R}}

\newcommand{\DD}{\mathcal{D}}
\newcommand{\BB}{\mathcal{B}}
\newcommand{\MM}{\mathcal{M}}
\newcommand{\KK}{\mathcal{K}}
\newcommand{\TT}{\mathcal{T}}
\newcommand{\EE}{\mathcal{E}}
\newcommand{\CC}{\mathcal{C}}
\newcommand{\FF}{\mathcal{F}}
\newcommand{\GG}{\mathcal{G}}
\newcommand{\XX}{\mathcal{X}}

\renewcommand{\SS}{\mathcal{S}}
\newcommand{\PP}{\mathcal{P}}

\newcommand{\JJ}{\mathcal{J}}
\newcommand{\N}{\mathbb{N}}
\newcommand{\NN}{\mathbb{N}}

\DeclareMathOperator{\tw}{\mathsf{tw}}

\newcommand{\blah}[2]{{#1}^{\langle{#2}\rangle}}
\newcommand{\compdeg}[3]{\overline{\deg}_{#1,#2}(#3)}

\begin{document}

\title{Tree Densities in Sparse Graph Classes}

\author[T.~Huynh]{Tony Huynh}
\author[D.R.~Wood]{David R. Wood}
\address[T.~Huynh and D.R.~Wood]{\newline School of Mathematics
\newline Monash University
\newline Melbourne, Australia}
\email{\{tony.bourbaki@gmail.com, david.wood@monash.edu\}}

\thanks{Research supported by the Australian Research Council.}

\date{\today}
\sloppy

\begin{abstract}
What is the maximum number of copies of a fixed forest $T$ in an $n$-vertex graph in a graph class $\GG$ as $n\to \infty$? We answer this question for a variety of sparse graph classes $\GG$. In particular, we show that the answer is $\Theta(n^{\alpha_d(T)})$ where $\alpha_d(T)$ is the size of the largest stable set in the subforest of $T$ induced by the vertices of degree at most $d$, for some integer $d$ that depends on $\GG$. For example, when $\GG$ is the class of $k$-degenerate graphs then $d=k$; when $\GG$ is the class of graphs containing no $K_{s,t}$-minor ($t\geq s$) then $d=s-1$; and when $\GG$ is the class of $k$-planar graphs then $d=2$. All these results are in fact consequences of a single lemma in terms of a finite set of excluded subgraphs. 
\end{abstract}

\maketitle 


\section{Introduction}
 
Many classical theorems in extremal graph theory concern the maximum number of copies of a fixed graph $H$ in an $n$-vertex graph\footnote{All graphs in this paper are undirected, finite, and simple, unless stated otherwise. Let $\NN:=\{1,2,\dots\}$ and $\NN_0:=\NN\cup\{0\}$. For $a,b\in\NN_0$, let $[a,b]:=\{a,a+1,\dots,b\}$ and $[b]:=[1,b]$. } in some class $\GG$. Here, a \emph{copy} means a subgraph isomorphic to $H$. For example, Tur\'an's Theorem determines the maximum number of copies of $K_2$ (that is, edges) in an $n$-vertex $K_t$-free graph~\citep{Turan41}. More generally, Zykov's Theorem determines the maximum number of copies of a given complete graph $K_s$ in an $n$-vertex $K_t$-free graph~\citep{Zykov49}. The excluded graph need not be complete. The Erd\H{o}s--Stone Theorem~\citep{ES46}  determines, for every non-bipartite graph $X$, the asymptotic maximum number of copies of $K_2$ in an $n$-vertex graph with no $X$-subgraph. Analogues of the Erd\H{o}s--Stone Theorem for the number of (induced) copies of a given graph within a graph class defined by an excluded (induced) subgraph have recently been widely studied \citep{AS19,AS16,AKS18,GGMV20,GMV19,GP19,GSTZ19,EMSG19,NesOss11,MQ20,Letzter19}.


For graphs $H$ and $G$, let $C(H,G)$ be the number of copies of $H$ in $G$. For a  graph class $\GG$, let 
$$C(H,\GG,n) := \max_{G\in\GG,\,|V(G)|=n} C(H,G).$$

This paper determines  the asymptotic behaviour of $C(T,\GG,n)$ as $n\to \infty$ for various sparse graph classes $\GG$ and for an arbitrary fixed forest $T$. In particular, we show that $C(T,\GG,n) \in \Theta(n^k)$ for some $k$ depending on $T$ and $\GG$. 

It turns out that $k$ depends on the size of particular stable sets in $T$. A set $S$ of vertices in a graph $G$ is \emph{stable} if no two vertices in $S$ are adjacent. Let $\alpha(G)$ be the size of a largest stable set in $G$. For a graph $G$ and $s\in\NN_0$, let $$\alpha_s(G):= \alpha( G[\{ v\in V(G): \deg_G(v)\leq s\}]).$$ 
Note that for a forest $T$ (indeed any bipartite graph), $\alpha_s(T)$ can be computed in polynomial time. See \citep{BW05,BSW-DM04,BDW-CDM06} for bounds on the size of bounded degree stable sets in forests, planar graphs, and other classes. 

The first sparse class we consider are the graphs of given degeneracy\footnote{A graph $G$ is \emph{$k$-degenerate} if every subgraph of $G$ has minimum degree at most $k$.}. 

\begin{thm}
\label{Degeneracy}
Fix $k\in\NN$ and let $\DD_k$ be the class of $k$-degenerate graphs. Then for every fixed forest $T$, 
$$C(T,\DD_k,n) \in \Theta(n^{\alpha_k(T)}).$$ 
\end{thm}

Our second main theorem determines $C(T,\GG,n)$ for many minor-closed classes\footnote{A graph $H$ is a \emph{minor} of a graph $G$ if a graph isomorphic to $H$ can be obtained from a subgraph of $G$ by contracting edges. A graph class $\GG$ is \emph{minor-closed} if some graph is not in $\GG$, and for every graph $G\in \GG$, every minor of $G$ is also in $\GG$.}\textsuperscript{,}\footnote{A \emph{tree decomposition} of  a graph $G$ is given by a tree $T$ whose nodes index a collection $(B_x\subseteq V(G):x\in V(T))$ of sets of vertices in $G$ called  \emph{bags}, such that: (T1) for every edge $vw$ of $G$, some bag $B_x$ contains both $v$ and $w$, and (T2) for every vertex $v$ of $G$, the set $\{x\in V(T):v\in B_x\}$ induces a non-empty (connected) subtree of $T$. The \emph{width} of such a tree decomposition is $\max\{|B_x|-1:x\in V(T)\}$. The \emph{treewidth}  of a graph $G$, denoted by $\tw(G)$, is the minimum width of a tree decompositions of $G$. 
See~\citep{Reed97,HW17} for surveys on treewidth. For each $s\in \NN$ the class of graphs with treewidth at most $s$ is minor-closed.}. Several examples of this result are given in \cref{MinorExamples}. 

\begin{thm}
\label{MinorClosedClass}
Fix $s,t\in\NN$ and let $\GG$ be a minor-closed class such that every graph with treewidth at most $s$ is in $\GG$ and $K_{s+1,t}\not\in\GG$. Then for every fixed forest $T$, 
$$C(T,\GG,n) \in \Theta(n^{\alpha_s(T)}).$$ 
\end{thm}

The lower bounds in \cref{Degeneracy,MinorClosedClass} are proved via the same construction given in \cref{LowerBounds}. The upper bounds in \cref{Degeneracy,MinorClosedClass} are proved in \cref{UpperBounds}. We in fact prove  a stronger result (\cref{UpperBound}) that shows that for any fixed forest $T$ and $s\in\NN$ there is a particular finite set $\FF$ such that $C(T,G) \in O(n^{\alpha_s(T)})$ for every $n$-vertex graph $G$ with $O(n)$ edges and containing no subgraph in $\FF$. This result is applied in \cref{Beyond} to determine $C(T,\GG,n)$ for various non-minor-closed classes $\GG$. For example,
we show a $\Theta(n^{\alpha_2(T)})$ bound for graphs that can be drawn in a fixed surface with a bounded average number of crossings per edge, which matches the known bound with no crossings.  


\subsection{Related Results}

Before continuing we mention related results from the literature.
For a fixed complete graph $K_s$, $C(K_s,\GG,n)$ has been extensively studied for various  graph classes $\GG$ including: graphs of given maximum degree \citep{Chase20,CR14,EG14,Kahn01,ACM12,Galvin11,GLS15,Wood-GC07,CR17}; graphs with a given number of edges, or more generally, a given number of smaller complete graphs~\citep{CR11,Frohmader10,Eckhoff-DM04,Eckhoff-DM99,FR90,KR19,KN75,FR92,PR00,Hedman-DM85}; graphs without long cycles~\citep{Luo18}; planar graphs~\citep{HS79,Wood-GC07,PY-IPL81}; graphs with given Euler genus~\citep{DFJSW,HJW20}; and graphs excluding a fixed minor or subdivision~\citep{ReedWood-TALG,NSTW-JCTB06,FOT,LO15,FW17,FW20}. 


When $\JJ$ is the class of planar graphs, $C(H,\JJ,n)$ has been determined for various graphs $H$ including: complete bipartite graphs \citep{AC84}, planar triangulations without non-facial triangles~\citep{AC84}, triangles \citep{HS79,HS82,HHS01,Wood-GC07}, 4-cycles \citep{HS79,Alameddine80}, 5-cycles \citep{GPSTZb}, 4-vertex paths \citep{GPSTZa}, and 4-vertex complete graphs \citep{AC84,Wood-GC07}. $C(H,\JJ,n)$ has also been studied for more general planar graphs $H$. Perles (see \citep{AC84}) conjectured that if $H$ is a fixed 3-connected planar graph, then $C(H,\mathbb{S}_0,n) \in \Theta(n)$. Perles noted the converse: If $H$ is planar, not 3-connected and $|V(H)|\geq 4$, then $C(H,\mathbb{S}_0,n) \in \Omega(n^2)$.  Perles' conjecture was proved by \citet{Wormald86} and independently by \citet{Eppstein93}, 
Recently, \citet{HJW20} extended these results to all surfaces and all graphs $H$ (see \cref{MinorExamples}). 

Finally, we mention a result of \citet{NesOss11}, who proved that for every infinite nowhere dense hereditary graph class $\GG$ and for every fixed graph $F$, the maximum, taken over all $n$-vertex graphs $G\in\GG$, of the number of induced subgraphs of $G$ isomorphic to $F$ is $\Omega( n^{\beta})$ and $O( n^{\beta+o(1)})$ for some integer $\beta\leq \alpha(F)$. Our results (when $F$ is a forest and $\GG$ is one of the classes that we consider) imply this upper bound  (since the number of induced copies of $T$ in $G$ is at most $C(T,G)$). Moreover, our bounds are often more precise since $\alpha_s(T)$ can be significantly less than $\alpha(T)$. 



\section{Lower Bound}
\label{LowerBounds}

\begin{lem}
\label{LowerBound}
Fix $s\in\NN$ and let $T$ be a fixed forest with $\alpha_s(T)=k$. Then there exists a constant $c_{\ref{LowerBound}}(k):= (2k)^{-k}$ such that for all sufficiently large $n\in\mathbb{N}$, there exists a graph $G$ with $|V(G)|\leq n$ and $\tw(G)\leq s$ and $C(T,G)\geq c_{\ref{LowerBound}}(k) n^k$. 
\end{lem}

\begin{proof}
Let $S$ be a maximum stable set in $T[\{ v\in V(T): \deg_T(v)\leq s\}]$ with $|S|=k$. Let $m:=\floor{\frac{n-|V(T)|}{k}}$. Let $G$ be the graph obtained from $T$ as follows: for each vertex $v$ in $S$ add to $G$ a set $C_v$ of $m$ vertices, such that $N_G(x):= N_T(v)$ for each vertex $x\in C_v$. Observe that $G$ has at most $n$ vertices. Each choice of one vertex $x\in C_v$ (for each $v\in S$), along with the vertices in $V(T)\setminus S$, induces a copy of $T$. Thus $C(T,G)\geq m^k$, which is at least $c_{\ref{LowerBound}}(k) n^k$ for $n\geq 2|V(T)| + 2k$. 


We now show $\tw(G) \leq s$. Let $T_1$ be a connected component of $T$ and $G_1$ be the corresponding connected component of $G$. Since the treewidth of a graph equals the maximum treewidth of its components, it suffices to show $\tw(G_1) \leq s$. We may assume $|V(T_1)| \geq 2$, as otherwise $\tw(G_1)=0$.  Let $T_1'$ be the tree obtained from $T_1$ as follows: for each vertex $v\in S \cap V(T_1)$ and each vertex $x\in C_v$, add one new vertex $x$ and one new edge $xv$ to $T_1'$.  Choose $r \in V(T_1) \setminus S$ and consider $T_1'$ to be rooted at $r$.   We  use $T_1'$ to define a tree-decomposition of $G_1$, where the bags are defined as follows.  Let $B_r:=\{r\}$. For each vertex $w\in V(T_1)\setminus(S\cup\{r\})$, if $p$ is the parent of $w$ in $T_1'$, let $B_w:=\{w,p\}$. For each vertex $v\in S \cap V(T_1)$ and each vertex $x$ in $C_v$, let $B_v:=N_{T_1}(v) \cup\{v\}$ and $B_x := N_{T_1}(v) \cup\{x\}$. 

We now show that $(B_x:x\in V(T_1'))$ is a tree-decomposition of $G_1$. 
The bags containing $r$ are indexed by $N_{T_1}(r)\cup\{r\}$, which induces a (connected) subtree of $T_1'$. For each vertex $w\in V(T_1)\setminus(S\cup\{r\})$ with parent $p$, the bags containing $w$ are those indexed by $\cup \{ C_v\cup\{v\} : v \in N_{T_1}(w) \cap S \}\cup\{w\}\cup (N_{T_1}(w)\setminus\{p\})$, which induces a subtree of $T_1'$ (since $vx\in E(T_1')$ for each $x\in C_v$ where $v\in N_{T_1}(w)\cap S$). 
For each vertex $v\in S$ with parent $p$, the bags containing $v$ are those indexed by $N_{T_1}(v) \cup\{v\} \setminus \{p\}$, which induces a subtree of $T_1'$. 
For each vertex $v\in S$ and $x\in C_v$, $B_x$ is the only bag that contains $x$. Hence propery (T1) in the definition of tree-decomposition  holds. For each edge $pv$ of $T_1$ where $p$ is the parent of $v$, the bag $B_v$ contains both $p$ and $v$. Every other edge of $G_1$ joins $x$ and $w$ for some $v\in S$ and $x\in C_v$ and $w\in N_{T_1}(v)$, in which case $B_x$ contains both $x$ and $w$. Hence (T2) holds. Therefore $(B_x:x\in V(T_1'))$ is a tree-decomposition of $G_1$. Since each bag has size at most $s+1$, we have $\tw(G_1)\leq s$.
\end{proof}

\section{Upper Bound}
\label{UpperBounds}

To prove upper bounds on $C(T,\GG,n)$, it is convenient to work in the following setting. For graphs $G$ and $H$, an  \emph{image} of $H$ in $G$ is an injection $\phi: V(H) \to V(G)$ such that $\phi(u)\phi(v) \in E(G)$ for all $uv \in E(H)$. Let  $I(H,G)$ be the number of images of $H$ in $G$. For a graph class $\GG$, let  $I(H,\GG,n)$ be the maximum of $I(H,G)$ taken over all $n$-vertex graphs $G\in\GG$.
If $H$ is fixed then $C(H,G)$ and $I(H,G)$ differ by a constant factor. In particular, if $|V(H)|=h$ then 
\begin{align}
C(H,G)& \leq I(H,G) \leq h!\, C(H,G), \nonumber \\
\label{CIC}
C(H,\GG,n)& \leq I(H,\GG,n) \leq h!\, C(H,\GG,n). 
\end{align}
So to bound $C(T,\GG,n)$ it suffices to work with images rather than copies. 

Our proof needs two tools from the literature. The first is due to  \citet{Eppstein93}. A collection $\mathcal{H}$ of images of a graph $H$ in a graph $G$ is \emph{coherent} if for all distinct images $\phi_1, \phi_2 \in \mathcal{H}$ and for all distinct vertices $x,y\in V(H)$, we have $\phi_1(x) \neq \phi_2(y)$.

\begin{lem}[\citep{Eppstein93}] 
\label{coherence}
Let $H$ be a graph with $h$ vertices and let $G$ be a graph.  Every collection of at least $c_{\ref{coherence}}(h,t):=h!^2 t^h$ images of $H$ in $G$ contains a coherent subcollection of size at least $t$.  
\end{lem}

We also use the following result of \citet{ER60}; see~\citep{ALWZ,BCW21} for recent quantitative improvements. A \emph{$t$-sunflower} is a collection $\mathcal{S}$ of $t$ sets for which there exists a set $R$ such that $X\cap Y=R$ for all distinct $X,Y\in\mathcal{S}$. The set $R$ is called the \emph{kernel} of $\mathcal{S}$. 

\begin{lem}[Sunflower Lemma~\citep{ER60}] \label{sunflower}
Every collection of at least $c_{\ref{sunflower}}(h,t):=h!(t-1)^h +1$ many $h$-subsets of a set contains a $t$-sunflower.  
\end{lem}

Consider graphs $H$ and $G$. An \emph{$H$-model} in a graph $G$ is a collection $(X_v:v\in V(H))$ of pairwise disjoint connected subgraphs of $G$ indexed by the vertices of $H$, such that for each edge $vw\in E(H)$ there is an edge of $G$ joining $X_v$ and $X_w$. Each subgraph $X_v$ is called a \emph{branch set}. A graph $G$ contains an $H$-model if and only if $H$ is a minor of $G$. An  $H$-model $(X_v:v\in V(H))$ in $G$ is \emph{$c$-shallow} if $X_v$ has radius at most $c$ for each $v\in V(H)$. An $H$-model $(X_v:v\in V(H))$  in $G$ is \emph{$c$-small} if $|V(X_v)|\leq c$ for each $v\in V(H)$. Shallow models are key components in the sparsity theory of \citet{Sparsity}. Small models have also been studied \citep{CHJR19,FJTW12,Montgomery15,SS15}. 

The next lemma is the heart of the paper. To describe the result we need the following construction, illustrated in \cref{Construction}. For a graph $H$, and $s,t\in\NN$, and $v\in V(H)$ let 
$$\compdeg{H}{s}{v} := \max\{s+1-\deg_H(v),0\}.$$ 
Then define  $\blah{H}{s,t}$ to be the graph with vertex set 
\begin{align*}
V( \blah{H}{s,t}) := \,
& \{(v,i):v\in V(H),i\in[t]\} \; \cup\\
& \{(v,j)^\star:v\in V(H),j\in[\compdeg{H}{s}{v}]\}
\end{align*}
and edge set
\begin{align*}
E( \blah{H}{s,t}) := \,
& \{(v,i)(w,i):vw\in E(H),i\in[t]\}\; \cup\\
& \{(v,i)(v,j)^\star:v\in V(H),i\in[t],j\in[\compdeg{H}{s}{v}]\}.
\end{align*}

\begin{figure}[h]
    \centering
    \includegraphics[width=\textwidth]{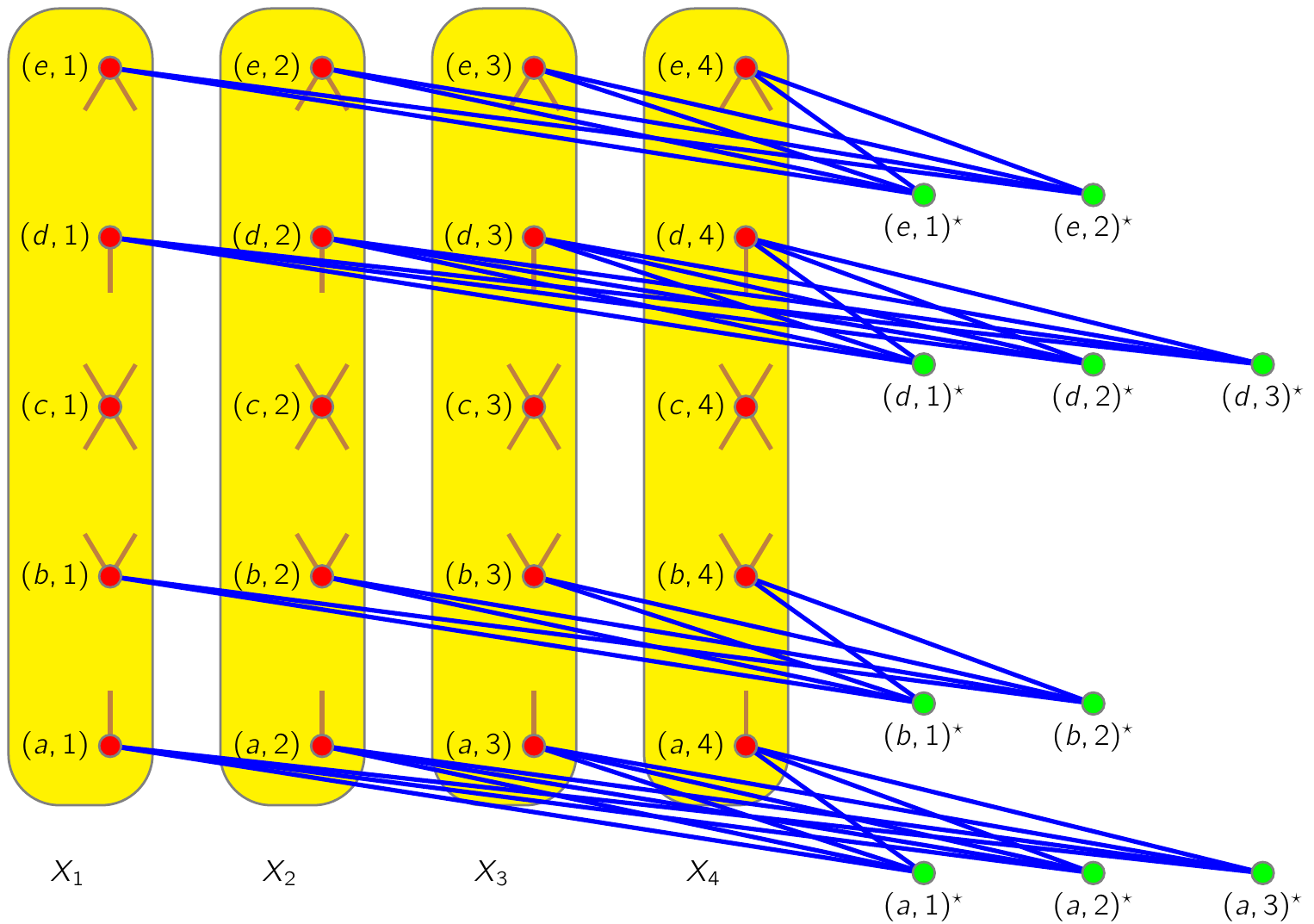}
    \vspace*{-3ex}
    \caption{$\blah{H}{3,4}$ where $V(H)=\{a,b,c,d,e\}$.    \label{Construction}}
\end{figure}

Several notes about $\blah{H}{s,t}$ are in order:
\begin{enumerate}[(A)] 
\item For each $i\in[t]$, let $X_i$ be the subgraph of $\blah{H}{s,t}$ induced by $\{(v,i):v\in V(H)\}$. Then $X_i\cong H$. 
Contracting each $X_i$ to a single  vertex produces $K_{s',t}$ where 
$$s':=\!\! \sum_{v\in V(H)} \!\! \compdeg{H}{s}{v} 
\geq \!\! \sum_{v\in V(H)} \!\! ( s+1-\deg_H(v) ) 
= (s+1)\,|V(H)| -2|E(H)|.$$
If $H$ is a non-empty tree then $s'\geq  |V(H)|(s-1)+2\geq s+1$, implying $K_{s+1,t}$ is a minor of $\blah{H}{s,t}$. 
\item Each vertex $(v,j)^\star$ has degree $t$ and each vertex $(v,i)$ has degree $\deg_H(v)+\compdeg{H}{s}{v}\geq s+1$. In particular, if $t\geq s+1$ then $\blah{H}{s,t}$ has minimum degree at least $s+1$. 
\item If $H$ is connected then $\text{diameter}(\blah{H}{s,t})\leq\text{diameter}(H)+2$.
\end{enumerate}

Define the \emph{density} of a graph $G$ to be $\rho(G):=\frac{|E(G)|}{|V(G)|}$. For a graph class $\GG$, let $\rho(\GG):=\sup\{ \rho(G): G\in \GG\}$

\begin{lem}
\label{UpperBound}
For all $s,t, h \in \N$ and $\rho \in \RR_{\geq0}$, there exists a constant $c:=c_{\ref{UpperBound}}(s,t,h,\rho) := c_{\ref{coherence}}( h, c_{\ref{sunflower}}(h,t))\, (\rho+1)^h$ 
such that for every forest $T$ with $h$ vertices, if $G$ is a graph with $\rho(G)\leq \rho$ and $I(T,G) \geq c\,|V(G)|^{\alpha_s(T)}$, then $G$ contains $\blah{U}{s,t}$ as a subgraph for some (non-empty) subtree $U$ of $T$. 
\end{lem}

\begin{proof}
Let $S:=\{ v \in V(T): \deg_T(v)\leq s\}$. Let $X$ be a stable set in $F:=T[S]$ of size $k:=\alpha_s(T)$. Since $F$ is bipartite, by Konig's Edge Cover Theorem~\citep{Konig36}, there is a set $Y\subseteq V(F)\cup E(F)$ with $|Y|=|X|$ such that each vertex of $F$ is either in $Y$ or is incident to an edge in $Y$. In fact, $Y\cap V(F)$ is the set of isolated vertices of $F$, although we will not need this property. 

Let $G$ be an $n$-vertex graph with  $\rho(G)\leq \rho$ and $I(T,G) \geq c\,n^k$. Let $\mathcal{I}$ be the set of images of $T$ in $G$. So $|\mathcal{I}|\geq  c\,n^{k}$. Let $\mathcal{X}  := \binom{ V(G) \cup E(G)}{k}$. Note that $|\mathcal{X}| \leq \binom{(\rho+1)n}{k} \leq (\rho+1)^k n^k$. For each $\phi\in\mathcal{I}$, let 
$$Y_\phi := \{ \phi(x):x\in Y\cap V(F)\} \cup \{ \phi(x)\phi(y) : xy \in Y \cap E(F)\},$$ 
which is an element of $\mathcal{X}$ since $|Y|=k$.
For each $Z\in\mathcal{X}$, let $\mathcal{I}_Z:= \{\phi\in\mathcal{I}: Y_\phi=Z\}$. By the pigeonhole principle, there exists $Z\in\mathcal{X}$ such that 
$$|\mathcal{I}_Z|\geq 
|\mathcal{I}| / |\mathcal{X}| \geq 
c /(\rho+1)^k \geq  
c /(\rho+1)^h =
c_{\ref{coherence}}( h, c_{\ref{sunflower}}(h,t)).$$ 
By \cref{coherence} applied to $\mathcal{I}_Z$, there is a coherent family $\mathcal{I}_1\subseteq \mathcal{I}_Z$ with $|\mathcal{I}_1|=c_{\ref{sunflower}}(h,t)$. 

We claim that the vertex sets in $G$ corresponding to the images of $T$ in $\mathcal{I}_1$ are all distinct. Suppose that $V(\phi_1(V(T)))=V(\phi_2(V(T)))$ for $\phi_1,\phi_2\in \mathcal{I}_1$. Let $x$ be any vertex in $T$. If $\phi_1(x)\neq\phi_2(x)$, then $\phi_2(y)=\phi_1(x)$ for some vertex $y$ of $T$ with $y\neq x$ (since $V(\phi_1(V(T)))=V(\phi_2(V(T)))$), which contradicts the definition of coherence. Thus $\phi_1(x)=\phi_2(x)$ for each vertex $x$ of $T$. Thus $\phi_1=\phi_2$. This proves our claim.  

Therefore, by \cref{sunflower} applied to $\{ \phi(V(T)) : \phi\in\mathcal{I}_1\}$, there is a set $R$ of vertices in $G$ and a subfamily $\mathcal{I}_2 \subseteq \mathcal{I}_1$ such that $\phi_1(V(T))\cap \phi_2(V(T))=R$ for all distinct $\phi_1, \phi_2 \in \mathcal{I}_2$, and $|\mathcal{I}_2|=t$. 

Fix $\phi_0 \in \mathcal{I}_2$ and let $K:=\phi_0^{-1}(R)$. Note that $K$ does not depend on the choice of $\phi_0$. Moreover,
$S \subseteq K$ because $Y_\phi=Z$ for every $\phi\in \mathcal{I}_2$, and each vertex in $S$ is either in $Y$ or is incident to an edge in $Y$. Let $U$ be some connected component of $T - K$. Note that $V(U) \cap S=\emptyset$, since $S \subseteq K$.   Thus  each vertex $v\in V(U)$ has  $\deg_T(v)\geq s+1$ and thus there is a set $N_v$ of at least $\compdeg{U}{s}{v}$ neighbours of $v$ in $K$. Again by coherence, $\phi_1(N_v)=\phi_2(N_v)$ for all $\phi_1,\phi_2\in\mathcal{I}_2$ and $v\in V(U)$. Observe that $N_{v_1}\cap N_{v_2}=\emptyset$ for distinct $v_1, v_2 \in U$, as otherwise $T$ would contain a cycle. Thus  $(\phi(U):\phi\in\mathcal{I}_2)$ and $(\phi_0(N_v):v\in V(U))$ define a subgraph of $G$ isomorphic to $\blah{U}{s,t}$. 
\end{proof}

We now prove our first main result. 

\begin{proof}[Proof of \cref{Degeneracy}] 
Since every graph with treewidth $k$ is in $\DD_k$, \cref{LowerBound} implies $C(T,\DD_k,n)\in \Omega(n^{\alpha_k(T)})$. For the upper bound, let $G$ be a $k$-degenerate graph. So $\rho(G)\leq k$. By \cref{UpperBound} with $s=k$ and $t=k+1$, if $I(T,G) \geq c\,|V(G)|^k$ then $G$ contains $\blah{U}{k,k+1}$ as a subgraph for some subtree $U$ of $T$. However, $\blah{U}{k,k+1}$ has minimum degree $k+1$, contradicting the $k$-degeneracy of $G$. Hence $I(T,G) \leq c\,|V(G)|^k$ and $C(T,\DD_k,n)\in O(n^{\alpha_k(T)})$ by \cref{CIC}.
\end{proof}

The following special case of \cref{UpperBound} will be useful. Say $\{X_1,\dots,X_s;Y_1,\dots,Y_t\}$ is a \emph{$(p,q)$-model} of $K_{s,t}$ in a graph $G$ if:
\begin{compactitem}
\item $X_1,\dots,X_s,Y_1,\dots,Y_t$ are pairwise disjoint connected subgraphs of $G$, 
\item for each $i\in[s]$ and $j\in[t]$ there is an edge in $G$ between $X_i$ and $Y_j$, 
\item $|V(X_i)|\leq p$ for each $i\in[s]$ and $|V(Y_j)|\leq q$ for each $j\in[t]$. 
\end{compactitem}

\begin{cor}
\label{UpperBoundCorollary}
For all $s,t, h \in \N$ and $\rho \in \RR_{\geq0}$, 
for every forest $T$ with $h$ vertices, 
if $G$ is a graph with $\rho(G)\leq \rho$ and 
$I(T,G) \geq c_{\ref{UpperBound}}(s,t,h,\rho) \,|V(G)|^{\alpha_s(T)}$, 
then for some $h' \in [h]$, $G$ contains a subgraph of diameter at most $h'+1$ that contains a $(1,h')$-model of $K_{h'(s-1)+2,t}$. In particular, $G$ contains a $(1,h)$-model of $K_{s+1,t}$. 
\end{cor}

\begin{proof}
By \cref{UpperBound},  $G$ contains $\blah{U}{s,t}$ as a subgraph for some subtree $U$ of $T$. The main claim follows from (A) and (C) where $h':=|V(U)|$. The final claim follows since $h'\in[h]$, implying $h'(s-1)+2 \geq s+1$. 
\end{proof}

\section{Minor-Closed Classes}
\label{MinorExamples}

\cref{MinorClosedClass} is implied by \cref{LowerBound,UpperBoundCorollary} and since every minor-closed class has bounded density \citep{Thomason84,Kostochka84}. We now give several examples of \cref{MinorClosedClass}. 

\subsection*{Treewidth:} 

Let $\mathcal{T}_k$ be the class of graphs with treewidth at most $k$. Then $\TT_k$ is a minor-closed class, and every graph in $\mathcal{T}_k$ has minimum degree at most $k$, implying $\rho(\TT_k)\leq k$ and $K_{k+1,k+1}\not\in \mathcal{T}_k$. Thus \cref{MinorClosedClass} with $s=k$ implies that for every fixed forest $T$, 
$$C(T,\mathcal{T}_k,n) \in \Theta(n^{\alpha_k(T)}). \label{Treewidth}$$

\subsection*{Surfaces:}

Let $\SS_{\Sigma}$ be  the class of graphs that embed\footnote{For $h\geq 0$, let $\mathbb{S}_h$ be the sphere with $h$ handles. For  $c\geq 0$, let $\mathbb{N}_c$ be the sphere with $c$ cross-caps. Every surface is homeomorphic to $\mathbb{S}_h$ or $\mathbb{N}_c$. The \emph{Euler genus} of $\mathbb{S}_h$ is $2h$. The \emph{Euler genus} of $\mathbb{N}_c$ is $c$. The \emph{Euler genus} of a graph $G$ is the minimum Euler genus of a surface in which $G$ embeds with no crossings. See~\citep{MoharThom} for background about graphs embedded in surfaces.} in a surface $\Sigma$. Then $\SS_{\Sigma}$ is a minor-closed class. \citet{HJW20} proved that for every $H\in\SS_\Sigma$,  
$$C(H, \SS_{\Sigma}, n)\in \Theta(n^{f(H)}),$$
where $f(H)$ is a graph invariant called the \emph{flap-number} of $H$, which is independent of $\Sigma$. 
\citet{HJW20} noted that $f(T)=\alpha_2(T)$ for a forest $T$. So, in particular, 
$$C(T, \SS_{\Sigma}, n) \in \Theta(n^{\alpha_2(T)}).$$
This result is also implied by \cref{MinorClosedClass} since 
for every surface $\Sigma$ of Euler genus $g$, 
Euler's formula implies that $K_{3,2g+3}$ is not in $\SS_\Sigma$ (first observed by \citet{Ringel65}), and 
$$\rho(\SS_\Sigma)\leq \rho_g:= \max\{3,\tfrac14 (5 + \sqrt{24g+1} \};$$ 
see \citep{OOW19} for a proof. 

\subsection*{Excluding a Complete Bipartite Minor:}

Let $\BB_{s,t}$ be the class of graphs containing no complete bipartite graph $K_{s,t}$ minor, where $t\geq s$. Since $K_{s,t}$ has treewidth $s$, every graph with treewidth at most $s-1$ is in $\mathcal{B}_{s,t}$. By \cref{MinorClosedClass}, for every fixed forest $T$, \begin{equation}
\label{ExcludedCompleteBipartiteMinor}
C(T,\BB_{s,t},n)\in \Theta(n^{\alpha_{s-1}(T)}).
\end{equation}
This answers affirmatively a question raised by \citet{HJW20}.  

\subsection*{Excluding a Complete Minor:}

Let $\CC_k$ be the class of graphs containing no complete graph $K_k$ minor. Then $K_{k-1,k-1}\not\in \CC_k$ (since contracting a $(k-2)$-edge matching in $K_{k-1,k-1}$ gives $K_k$). Every graph with treewidth at most $k-2$ is in $\CC_k$. Thus \cref{MinorClosedClass} with $s=k-2$ implies that for every fixed forest $T$, 
$$C(T,\CC_k,n) \in \Theta(n^{\alpha_{k-2}(T)}).$$

\subsection*{Colin de Verdi\'ere Number:}

The  Colin de Verdi\`ere parameter $\mu(G)$ is an important graph invariant introduced by \citet{CdV90,CdV93}; see~\citep{HLS,Schrijver97} for surveys. It is known that $\mu(G)\leq 1$ if and only if $G$ is a disjoint union of paths, $\mu(G)\leq 2$ if and only if $G$ is outerplanar, $\mu(G)\leq 3$ if and only if $G$ is planar, and $\mu(G)\leq 4$ if and only if $G$ is linklessly embeddable. 
Let $\mathcal{V}_k:=\{G:\mu(G)\leq k\}$. Then $\mathcal{V}_k$ is a minor-closed class~\citep{CdV90,CdV93}. \citet{GB11} proved that $\mu(G) \leq\tw(G)+1$. So every graph with treewidth at most $k-1$ is in $\mathcal{V}_k$. \citet{HLS} proved that $\mu(K_{s,t}) = s+1$ for $t\geq\max\{s,3\}$, so $K_{k,\max\{k,3\}} \not\in \mathcal{V}_k$. Thus  \cref{MinorClosedClass} with $s=k-1$ and $t=\max\{k,3\}$ implies that for every fixed forest $T$, 
\begin{equation}
\label{CdV}
C(T,\mathcal{V}_k,n) \in \Theta(n^{\alpha_{k-1}(T)}).
\end{equation}

\subsection*{Linkless Graphs:}

A graph is \emph{linklessly embeddable} if it has an embedding in $\mathbb{R}^3$ with no two linked cycles~\citep{Sachs83,RST93a}. Let $\mathcal{L}$ be the class of linklessly embeddable graphs. Then $\mathcal{L}$ is a minor-closed class whose minimal excluded minors are the so-called Petersen family~\citep{RST95}, which includes $K_6$, $K_{4,4}$ minus an edge, and the Petersen graph. As mentioned above, $\mathcal{L}=\mathcal{V}_4$. Thus \cref{CdV} with $k=4$ implies for every fixed forest $T$, $$C(T,\mathcal{L},n) \in \Theta(n^{\alpha_{3}(T)}).$$

\subsection*{Knotless Graphs:}

A graph is \emph{knotlessly embeddable} if it has an embedding in $\mathbb{R}^3$ in which every cycle forms a trivial knot; see~\citep{Alfonsin05} for a survey. Let $\KK$ be the class of knotlessly embeddable graphs. Then $\KK$ is a minor-closed class whose minimal excluded minors include $K_7$ and $K_{3,3,1,1}$ (see \citep{CG83,Foisy02}). More than 260 minimal excluded minors are known~\cite{GMN14}, but the full list of minimal excluded minors is unknown. Since $K_7\not\in\KK$, we have $\rho(\KK)\leq\rho(\CC_7)<5$ by a theorem of \citet{Mader68}. \citet{Shimabara88} proved that $K_{5,5}\not\in\KK$. By \cref{MinorClosedClass}, 
$$C(T,\KK,n)\in O(n^{\alpha_4(T)}).$$
This bound would be tight if every treewidth 4 graph is knotlessly embeddable, which is an open problem of independent interest.






The above results all depend on excluded complete bipartite minors. We now show that excluded complete bipartite minors determine $C(T,\GG,n)$ for a broad family of minor-closed classes.

\begin{thm}
\label{BiconnectedForbiddenMinors}
Let $\GG$ be a minor-closed class such that every minimal forbidden minor of $\GG$ is 2-connected. Let $s$ be the maximum integer such that $K_{s,t} \in \GG$ for every $t\in \NN$. Then for every forest $T$, 
$$C(T,\GG,n) = \Theta( n^{\alpha_s(T)} ).$$ 
\end{thm}

\begin{proof}
Note that the condition that every minimal forbidden minor of $\GG$ is 2-connected is equivalent to saying that $\GG$ is closed under the 1-sum operation (that is, if $G_1,G_2\in\GG$ and $|V(G_1\cap G_2)|\leq 1$, then $G_1\cup G_2\in \GG$).  

The proof of \cref{LowerBound} shows that for all sufficiently large $n\in\NN$ there exists an $n$-vertex graph $G$ with $C(T,G)\geq c n^{\alpha_s(T)}$, where $G$ is obtained from 1-sums of complete bipartite graphs  $K_{s',t}$ with $s'\leq s$. By the definition of $s$ and since $\GG$ is closed under 1-sums, $G\in\GG$. Thus $C(T,\GG,n) \in \Omega( n^{\alpha_s(T)} )$.

Now we prove the upper bound. Since $\GG$ is minor-closed, $\GG$ has bounded
density~\citep{Thomason84,Kostochka84}. By the definition of $s$, there exists $t\in\NN$ such that $K_{s+1,t}\not\in \GG$. By (A), we have $\blah{U}{s,t}\not\in\GG$ for every non-empty subtree $U$ of $T$. Thus $I(T,\GG,n)  \in O( n^{\alpha_s(T)})$ by \cref{UpperBound}.
\end{proof}

Note that minor-closed classes with bounded pathwidth (that is, those excluding a fixed forest as a minor \citep{BRST91}) are examples not covered by \cref{BiconnectedForbiddenMinors}. Determining $C(T,\GG_k,n)$, where $\GG_k$ is the class of pathwidth $k$ graphs, is an interesting open problem.

\section{Beyond Minor-Closed Classes}
\label{Beyond}

This section asymptotically determines $C(T,\GG,n)$ for several non-minor-closed graph classes $\GG$. 

\subsection{Shortcut Systems}

\citet{DMW} introduced the following definition which generalises the notion of shallow immersion~\citep{NesOss15} and provides a way to describe a graph class in terms of a simpler graph class. Then properties of the original class are (in some sense) transferred to the new class. Let $\PP$ be a set of non-trivial paths in a graph $G$. Each path $P\in\PP$ is called a \emph{shortcut}; if $P$ has endpoints $v$ and $w$ then it is a \emph{$vw$-shortcut}. Given a graph $G$ and a shortcut system $\PP$ for $G$, let $G^\PP$ be the simple supergraph of $G$ obtained by adding the edge $vw$ for each $vw$-shortcut in $\PP$. \citet{DMW} defined $\PP$ to be a \emph{$(k,d)$-shortcut system} (for $G$) if:
\begin{compactitem}
\item every path in $\mathcal{P}$ has length at most $k$, and
\item for every $v\in V(G)$, the number of paths in $\mathcal{P}$ that use $v$ as an internal vertex is at most $d$.
\end{compactitem}

We use the following variation. Say $\PP$ is a \emph{$(k,d)^\star$-shortcut system} (for $G$) if:
\begin{compactitem}
\item every path in $\mathcal{P}$ has length at most $k$, and
\item for every $v\in V(G)$, if $M_v$ is the set of vertices $u\in V(G)$ such that there exists a $uw$-shortcut in $\PP$ in which $v$ is an internal vertex, then $|M_v| \leq d$. \end{compactitem}

Clearly, every $(k,d)^\star$-shortcut system is a $(k,\binom{d}{2})$-shortcut system (since $G^\PP$ is simple), and every $(k,d)$-shortcut system is a $(k,2d)^\star$-shortcut system.

The next lemma shows that if $G^\PP$ contains a `small' model of a `large' complete bipartite graph, then so does $G$. 

\begin{lem}
\label{pqModelShortcut}
For all $s,t,d,k,p,q\in\NN$, let 
$s':= (d(k-1)(p-1) + 1)(s-1)+1$ and 
$t':= ( 2d(k-1)(s+q-1) + 1 )(t-1) + 1 + sd ( p+(k-1)(p-1) )$. 
Let $\PP$ be a $(k,d)^\star$-shortcut system for a graph $G$. 
If $G^{\PP}$ contains a $(p,q)$-model of $K_{s',t'}$, then $G$ contains a 
$( p+(k-1)(p-1),q+(k-1)(s+q-1))$-model of $K_{s,t}$.
\end{lem}

\begin{proof}
Let $(X_1,\dots,X_{s'};Y_1,\dots,Y_{t'})$ be a $(p,q)$-model of $K_{s',t'}$ in $G^{\PP}$. We may assume that each edge of $G$ is (a path of length 1) in $\PP$. Let $I:=[s']$ and $J:=[t']$. We may assume that $X_i$ and $Y_j$ are subtrees of $G^{\PP}$ for $i\in I$ and $j\in J$. 

Consider each $i\in I$. Let $C_i$ be the set of all vertices internal to some $uw$-shortcut with $uw\in E(X_i)$. Since $|E(X_i)| \leq p-1$, we have $|C_i|\leq (k-1)(p-1)$. For each $i\in I$, let $\hat{X}_i$ be the subgraph of $G$ induced by $V(X_i)\cup C_i$. By construction, $\hat{X}_i$ is connected and $|V(\hat{X}_i)| \leq p+(k-1)(p-1)$. 

Consider the graph $A$ with $V(A):=I$ where two vertices $i,i'\in V(A)$ are adjacent if $V(\hat{X_i}) \cap V(\hat{X_{i'}}) \neq\emptyset$. 
For each $ii' \in E(A)$, fix a vertex $v_{i,i'}$ in $V(\hat{X}_i)\cap V(\hat{X}_{i'})$, which is in $C_i \cup C_{i'}$ since $V(X_i) \cap V(X_{i'}) = \emptyset$. For $i\in I$ and $v \in C_i$, define $E_{v,i}$ to be the set of all edges $ii'\in E(A)$ with $v_{i,i'} = v$. If $ii'$ is in $E_{v,i}$ and $v\not \in X_{i'}$, then $|M_v\cap X_{i'}|\geq 2$. Also $|M_v\cap X_i|\geq 2$.  Since $v$ is in at most one $X_{i'}$, in total, $|M_v|\geq 2|E_{v,i}|$, implying $|E_{v,i}|\leq \frac{d}{2}$. Since $|I|=|V(A)|$ and $|C_i|\leq (k-1)(p-1)$, 
$$|E(A)| 
\leq \sum_{i\in I} \sum_{v\in C_i} |E_{v,i}| 
\leq \tfrac{d}{2}(k-1)(p-1)\, |V(A)|.$$
Thus $A$ has average degree at most $d(k-1)(p-1)$. By Tur\'an's Theorem, $A$ contains a stable set $I'$ of size $\ceil{ |I|/ ( d(k-1)(p-1) + 1 )} =s$. For distinct $i,i'\in I'$, the subgraphs $\hat{X}_i$ and $\hat{X}_{i'}$ are disjoint. Let $\mathcal{X}:=\bigcup_{i \in I'} V(\hat{X}_i)$. Note that $|\mathcal{X}| \leq s ( p+(k-1)(p-1) )$.  

Let $Z:=\bigcup_{x\in \mathcal{X}}M_x$. Then $|Z|\leq  sd ( p+(k-1)(p-1) )$.
Thus $Y_j$ intersects $Z$ for at most $sd ( p+(k-1)(p-1) )$ elements $j\in J$. 
Hence $J$ contains a subset $K$ of size $( 2d(k-1)(s+q-1) + 1 )(t-1) +1$ such that $V(Y_j) \cap Z=\emptyset$ for each $j\in K$. 





Consider each $j\in K$. Initialise $D_j:=\emptyset$. For each $i\in I'$, choose $x\in V(X_i)$ and $w\in V(Y_j)$ such that $xw \in E(G^{\mathcal P})$, and add all the internal vertices of the $xw$-shortcut $P \in \mathcal P$ to $D_j$.  For each edge $uw$ of $Y_j$, add all the internal vertices of the $uw$-shortcut $P\in \PP$ to $D_j$. Note that 
$$|D_j| \leq (k-1)|I'| + (k-1)|E(Y_j)| \leq (k-1)(s+q-1),$$
since $Y_j$ has at most $q-1$ edges.
Moreover, $D_j \cap \XX = \emptyset$ since $V(Y_j)\cap Z=\emptyset$. 

For each $j\in K$, let $\hat{Y}_j$ be the subgraph of $G$ induced by $V(Y_j)\cup D_j$. By construction, $\hat{Y}_j$ is connected with at most $q+(k-1)(s+q-1)$ vertices and is disjoint from $\XX$. 

Consider the graph $B$ with $V(B):=K$ where two vertices $j,j'\in V(B)$ are adjacent if  $V(\hat{Y_j}) \cap V(\hat{Y_{j'}}) \neq\emptyset$. For each $jj' \in E(B)$, fix a vertex $v_{j,j'}$ in $V(\hat{Y}_j)\cap V(\hat{Y}_{j'})$, which is in $D_j \cup D_{j'}$ since $V(Y_j) \cap V(Y_{j'}) = \emptyset$. For $j\in K$ and $v \in D_j$, define $E_{v,j}$ to be the set of all edges $jj'\in E(B)$ with $v_{j,j'} = v$.

We now bound $|E(B)|$.  If $jj'$ is in $E_{v,j}$ and $v\not \in Y_{j'}$, then $|M_v\cap Y_{j'}|\geq 1$. Also $|M_v\cap Y_j|\geq 1$. Since $v$ is in at most one $Y_{j'}$, in total, $|M_v| \geq |E_{v,j}|$, implying $|E_{v,j}| \leq d$.   Since $|K|=|V(B)|$ and $|D_j|\leq (k-1)(s+q-1)$, 
$$|E(B)| 
\leq \sum_{j\in K} \sum_{v\in D_j} |E_{v,j}| 
\leq d(k-1)(s+q-1)\, |V(B)|,$$

implying $B$ has average degree at most $2d(k-1)(s+q-1)$. By Tur\'an's Theorem, $B$ contains a stable set $L$ of size $\ceil{ |K|/ ( 2d(k-1)(s+q-1) + 1 )} =t$.

For distinct $j,{j'}\in L$, since $L$ is a stable set in $B$, $\hat{Y}_j$ and $\hat{Y}_{j'}$ are disjoint. For each $j\in L$, $Y_j$ and $\XX$ are disjoint by assumption, and $D_j$ and $\XX$ are disjoint by construction. Also, for each $i\in I'$ and $j\in L$, there is an edge between $\hat{X}_i$ and $\hat{Y}_j$ by construction. Thus $\{\hat{X}_i:i\in I'\}$ and $\{\hat{Y}_j:j \in L\}$ form a $( p+(k-1)(p-1),q+k(s+q-1))$-model of $K_{s,t}$ in $G$. 
\end{proof}

\cref{pqModelShortcut} with $p=1$ implies the following result. We emphasise that the value of $s$ does not change in the two models. 

\begin{cor}
\label{1qModelShortcut}
Fix $s,t,k,d,q\in\NN$. Let $t':= ( 2d(k-1)(s+q-1) + 1 )(t-1) + 1 + sd$. 
Let $\PP$ be a $(k,d)^\star$-shortcut system for a graph $G$. 
If $G^{\PP}$ contains a $(1,q)$-model of $K_{s,t'}$, 
then $G$ contains a $(1,q+(k-1)(s+q-1))$-model of $K_{s,t}$.
\end{cor}

\subsection{Low-Degree Squares of Graphs}

The above result on shortcut systems leads to the following extension of our results for minor-closed classes. For a graph $G$ and $d\in\NN$, let $G^{(d)}$ be the graph obtained from $G$ by adding a clique on $N_G(v)$ for each vertex $v\in V(G)$ with $\deg_G(v)\leq d$. (This definition incorporates and generalises the square of a graph with maximum degree $d$.)\ Note that $G^{(d)}=G^{\PP}$, where $\PP$ is the $(2,d)^\star$-shortcut system $\{uvw: v\in V(G);\deg_G(v)\leq d; u,w\in N_G(v); u\neq w\}$. For a graph class $\GG$, let $\GG^{(d)}:=\{G^{(d)}:G\in \GG\}$. Note that $\rho(G^{(d)}) \leq \rho(G)+\binom{d}{2}$. \cref{UpperBoundCorollary} and \cref{1qModelShortcut} with $k=2$ and $q=h$ imply:

\begin{cor}
\label{SquareGraph} 
Fix $s,t,d,h\in\NN$ and $\rho\in\RR_{\geq 0}$. 
Let $T$ be fixed forest with $h$ vertices. 
Let $t':= ( 2d(s+h-1) + 1 )(t-1) + 1 + sd$. 
Let $G$ be a graph with $\rho(G)\leq \rho$ and containing no $(1,2h+s-1)$-model of $K_{s,t}$. 
Then $G^{(d)}$ contains no $(1,h)$-model of $K_{s,t'}$, and 
$$C(T,G^{(d)}) \leq I(T,G^{(d)}) 
\leq c_{\ref{UpperBound}}(s-1,t',h,\rho+\tbinom{d}{2}) \,|V(G)|^{\alpha_{s-1}(T)}.$$
\end{cor}

With \cref{LowerBound} we have:

\begin{thm}
\label{SquareClass}
Fix $s,t,d,h\in\NN$ and $\rho\in\RR_{\geq 0}$. 
Let $T$ be fixed forest with $h$ vertices. 
Let $t':= ( 2d(s+h-1) + 1 )(t-1) + 1 + sd$. 
Let $\GG$ be a graph class such that $\rho(\GG)\leq \rho$, 
every graph with treewidth at most $s-1$ is in $\GG$, and
no graph in $\GG$ contains a $(1,2h+s-1)$-model of $K_{s,t}$. 
Then no graph in $\GG^{(d)}$ contains a $(1,h)$-model of $K_{s,t'}$, and 
$$C(T,\GG^{(d)},n) = \Theta( n^{\alpha_{s-1}(T)} ).$$
\end{thm}

\cref{SquareClass} is applicable to all the minor-closed classes discussed in \cref{MinorExamples}. For example, we have the following extension of \cref{ExcludedCompleteBipartiteMinor}. Recall that $\BB_{s,t}^{(d)}$ is the class of graphs $G^{(d)}$ where $G$ contains no $K_{s,t}$-minor. Then for every fixed forest $T$, 
$$C(T,\BB_{s,t}^{(d)},n) = \Theta( n^{\alpha_{s-1}(T)}).$$

\subsection{Map Graphs}

Map graphs are defined as follows. Start with a graph $G_0$ embedded in a surface $\Sigma$, with each face labelled a ``nation'' or a ``lake'', where each vertex of $G_0$ is incident with at most $d$ nations. Let $G$ be the graph whose vertices are the nations of $G_0$, where two vertices are adjacent in $G$ if the corresponding faces in $G_0$ share a vertex. Then $G$ is called a \emph{$(\Sigma,d)$-map graph}. A $(\mathbb{S}_0,d)$-map graph is called a (plane) \emph{$d$-map graph}; such graphs have been extensively studied \citep{FLS-SODA12,Chen07,DFHT05,CGP02,Chen01}. 
Let $\MM_{\Sigma,d}$ be the set of all $(\Sigma,d)$-map graphs. 
Since $\MM_{\Sigma,3}=\SS_\Sigma$ (see \citep{CGP02,DEW17}), map graphs provide a natural generalisation of graphs embeddable in a surface. 

Let $G\in\MM_{\Sigma,d}$ where $\Sigma$ has Euler genus $g$. Let $T$ be a fixed forest with $h$ vertices. \citet{DMW} proved that $G$ is a subgraph of $G_0^\PP$ for some graph $G_0\in\SS_\Sigma$ and some $(2,\frac12 d(d-3))$-shortcut system $\PP$ of $G_0$. Inspecting the proof in \citep{DMW} one observes that $\PP$ is a $(2,d)^\star$-shortcut system. In the plane case, \citet{Chen01} proved that $\rho(\MM_{\mathbb{S}_0,d})< d$. An analogous argument shows that $\rho(\MM_{\Sigma,d})\in O(d \sqrt{g+1})$. The same bound can also be concluded from \cref{gkCloseEdges}. 
Since $G_0$ contains no $K_{3,2g+3}$ minor, by \cref{1qModelShortcut}, 
for each $q\in\NN$, 
$G_0^{\PP}$ and thus $G$ contains no $(1,q)$-model of $K_{3,t'}$ 
where $t':= ( 2d(q+2) + 1 )(2g+2) +1 + 3d$. 
With $q=h$, \cref{UpperBoundCorollary} then implies that 
$C(T,G) \leq I(T,G) \leq c_{\ref{UpperBound}}(2,t',h,\rho) \,|V(G)|^{\alpha_2(T)}$. Hence 
$$C(T,\MM_{\Sigma,d},n) \in \Theta( n^{\alpha_2(T)} ),$$
where the lower bound follows from \cref{LowerBound} since every graph with treewidth 2 is planar and is thus a $(\Sigma,d)$-map graph. Also note the $q=1$ case above shows that 
$$K_{3,( 6d + 1 )(2g+2)+1 + 3d} \not\in \MM_{\Sigma,d}.$$

\subsection{Bounded Number of Crossings}

Here we consider drawings of graphs with a bounded number of crossings per edge. Throughout the paper, we assume that no three edges cross at a single point in a drawing of a graph. For a surface $\Sigma$ and $k\in \NN$, let $\SS_{\Sigma,k}$ be the class of graphs $G$ that have a drawing in $\Sigma$ such that each edge is in at most $k$ crossings. Since $\SS_{\Sigma,0}=\SS_\Sigma$, this class provides a natural generalisation of graphs embeddable in surfaces and is widely studied~\citep{PachToth-DCG02,DMW,OOW19}. Graphs in $\SS_{\mathbb{S}_0,k}$ are called \emph{$k$-planar}. The case $k=1$ is particularly important in the graph drawing literature; see \citep{KLM17} for a bibliography with over 100 references. 

Let $T$ be a fixed forest with $h$ vertices. Let $G\in \SS_{\Sigma,k}$ where $\Sigma$ has Euler genus $g$. \citet{DMW} noted that by replacing each crossing point by a dummy vertex we obtain a graph $G_0\in\SS_\Sigma$ such that $G$ is a subgraph of $G_0^\PP$ for some $(k+1,2)$-shortcut system $\PP$, which is a $(k+1,4)^\star$-shortcut system. 
Results of \citet{OOW19} show that $\rho(\SS_{\Sigma,k}) \leq 2\sqrt{k+1}\rho_g$ (see \cref{gkCloseEdges} below).
Since $G_0$ contains no $K_{3,2g+3}$ minor, by \cref{1qModelShortcut}, 
for all $q\in\NN$, $G_0^\PP$ and thus $G$ contains no $(1,q)$-model of 
$K_{3,t'}$ where $t':= ( 8k(q+2) + 1 )(2g+2) + 13$. Applying this result with $q=h$, \cref{UpperBoundCorollary} then implies
$C(T,G) \leq I(T,G) \leq c_{\ref{UpperBound}}(2,t',h,2\sqrt{k+1}\rho_g) \,|V(G)|^{\alpha_{2}(T)}$. Hence  
\begin{equation}
\label{gkPlanar}
C(T,\SS_{\Sigma,k},n) \in \Theta(n^{\alpha_2(T)}), 
\end{equation}
where the lower bound follows from \cref{LowerBound} since every treewidth 2 graph is planar and is thus in $\SS_{\Sigma,k}$. Also note the $q=1$ case above shows that 
$$K_{3,( 24k + 1 )(2g+2) + 13} \not\in \SS_{\Sigma,k}.$$





\subsection{Bounded Average Number of Crossings}
\label{Crossings}

Here we generalise the results from the previous section for graphs that can be drawn with a bounded average number of crossings per edge. \citet{OOW19} defined a graph $G$ to be \emph{$k$-close to Euler genus $g$} if every subgraph $G'$ of $G$ has a drawing in a surface of Euler genus at most $g$ with at most $k\,|E(G')|$ crossings\footnote{The case $g=0$ is similar to other definitions from the literature, as we now explain. \citet{EG17} defined the \emph{crossing graph} of a drawing of a graph $G$ to be the graph with vertex set $E(G)$, where two vertices are adjacent if the corresponding edges in $G$ cross. \citet{EG17} defined a graph to be a \emph{$d$-degenerate crossing graph} if it admits a drawing whose crossing graph is $d$-degenerate. Independently, \citet{GapPlanar18}  defined a graph $G$ to be \emph{$k$-gap-planar} if $G$  has a drawing in the plane in which each crossing is assigned to one of the two involved edges and each edge is assigned at most $k$ of its crossings. This is equivalent to saying that the crossing graph has an orientation with outdegree at most $k$ at every vertex. \citet{Hakimi65} proved that any graph $H$ has such an orientation if and only if every subgraph of $H$ has average degree at most $2k$. So a graph $G$ is $k$-gap-planar if and only if $G$ has a drawing such that every subgraph of the crossing graph has average degree at most $2k$ if and only if $G$ has a drawing such that every subgraph $G'$ of $G$ has at most $k\,|E(G')|$ crossings in the induced drawing of $G'$. 
The only difference between ``$k$-close to planar'' and ``$k$-gap planar'' is that a $k$-gap planar graph has a single drawing in which every subgraph has the desired number of crossings. To complete the comparison, the definition of \citet{EG17} is equivalent to saying that $G$ has a drawing in which the crossing graph has an acyclic orientation with outdegree at most $k$ at every vertex. Thus every $k$-degenerate crossing graph is $k$-gap-planar graph, and every $k$-gap-planar graph is a $2k$-degenerate crossing graph. 
}. Let $\EE_{g,k}$ be the class of graphs $k$-close to Euler genus $g$. 
This is a broader class than $\SS_{\Sigma,k}$  since it allows an average of $k$ crossings per edge, whereas $\SS_{\Sigma,k}$ requires a maximum of $k$ crossings per edge. In particular, if $\Sigma$ has Euler genus $g$, then $\SS_{\Sigma,k} \subseteq \EE_{g,k/2}$.

The next lemma is of independent interest.

\begin{lem}
\label{gkCloseShallow}
Fix $g,r\in\NN_0$ and $d\in\NN$ and $k\in\RR_{\geq 0}$. Assume that graph $G\in\EE_{g,k}$ contains an  $r$-shallow $H$-model $(X_v:v\in V(H))$ such that for every vertex $v\in V(H)$ we have $\deg_H(v)\leq d$ or $|V(X_v)|=1$. Then $H$ is in $\EE_{g,2kd^2(2r+1)}$. 
\end{lem}

\begin{proof}
For each $v\in V(H)$, let $a_v$ be the central vertex of $X_v$. We may assume that $X_v$ is a BFS spanning tree of $G[V(X_v)]$ rooted at $a_v$ and with radius at most $r$. Orient the edges of $X_v$ away from $a_v$. 

Let $H'$ be an arbitrary subgraph of $H$. For each $v\in V(H')$, let $X'_v$ be a minimal subtree of $X_v$ rooted at $a_v$, such that $(X'_v:v\in V(H'))$ is an $r$-shallow $H'$-model. By minimality, $X'_v$ has at most $\deg_{H'}(v)$ leaves. Each edge of $X'_v$ is on a path from a leaf to $a_v$, implying $|E(X'_v)| \leq r\deg_{H'}(v)$. 

Let $G'$ be the subgraph of $G$ consisting of $\bigcup_{v\in V(H')}X'_v$ along with one undirected edge $y_{vw}y_{wv}$ for each edge $vw\in E(H')$, where $y_{vw}\in V(X'_v)$ and $y_{wv}\in V(X'_w)$. Let $P_{vw}$ be the directed $a_vy_{vw}$-path in $X'_v$.  Note that  $$|E(G')| 
\,=\,
|E(H')| + \!\! \sum_{v\in V(H')} \!\!\! |E(X'_v)| 
\,\leq\, 
|E(H')| + r \!\! \sum_{v\in V(H')} \!\!\! \deg_{H'}(v) 
\,=\,  (2r+1) |E(H')| . $$

Since $G$ is $k$-close to Euler genus $g$, $G'$ has a drawing in a surface of Euler genus at most $g$ with at most $k\,|E(G')|$ crossings. For each $e\in E(G')$, let $\ell(e)$ be the number of crossings on $e$ in this drawing of $G'$. Since each crossing contributes towards $\ell$ for exactly two edges,  $$\sum_{e\in E(G')}\!\!\ell(e) \leq 2k\,|E(G')| \leq 2k(2r+1)|E(H')| .$$ 

Let $G''$ be the multigraph obtained from $G'$ as follows: for each vertex $v\in V(H')$ and edge $e$ in $X'_v$, let the multiplicity of $e$ in $G''$ equal the number of edges $vw\in E(H')$ for which the path $P_{vw}$ uses $e$. Edges of $G''$ inherit their orientation from $G'$. Note that $G''$ has multiplicity at most $d$. By replicating edges in the drawing of $G'$ we obtain a drawing of $G''$ such that every edge of $G''$ corresponding to $e\in E(G')$ is in at most $d\, \ell(e)$ crossings. Since each edge $e\in E(G')$ has multiplicity at most $d$ in $G''$, the number of crossings in the drawing of $G''$ is at most 
$\sum_{e\in E(G')} d^2\ell(e) \leq 2kd^2(2r+1)\,|E(H')|$. 


Note that at each vertex $y$ in $G''$, in the circular ordering of edges in $G''$ incident to $y$ determined by the drawing of $G''$, all the incoming edges form an interval. We now use the drawing of $G'$ to produce a drawing of a graph $G'''$, which is a subdivision of $H'$, where each vertex $v\in V(H')$ is drawn at the location of $a_v$. Here is the idea (see \cref{gkCloseH}): First `assign' each edge $y_{vw}y_{wv}$ of $G'$ to the edge $vw$ of $H'$. Next `assign' each edge of $G'$ arising from some $X'_v$ to exactly one edge incident to $v$, such that for each edge $vw$ of $H'$ incident to $v$ there is a path in $G'$ from $a_v$ to $y_{vw}$ consisting of edges assigned to $vw$. Then each edge $vw$ in $H'$ is drawn by following this path. 

\begin{figure}[h]
\centering
\includegraphics[width=\textwidth]{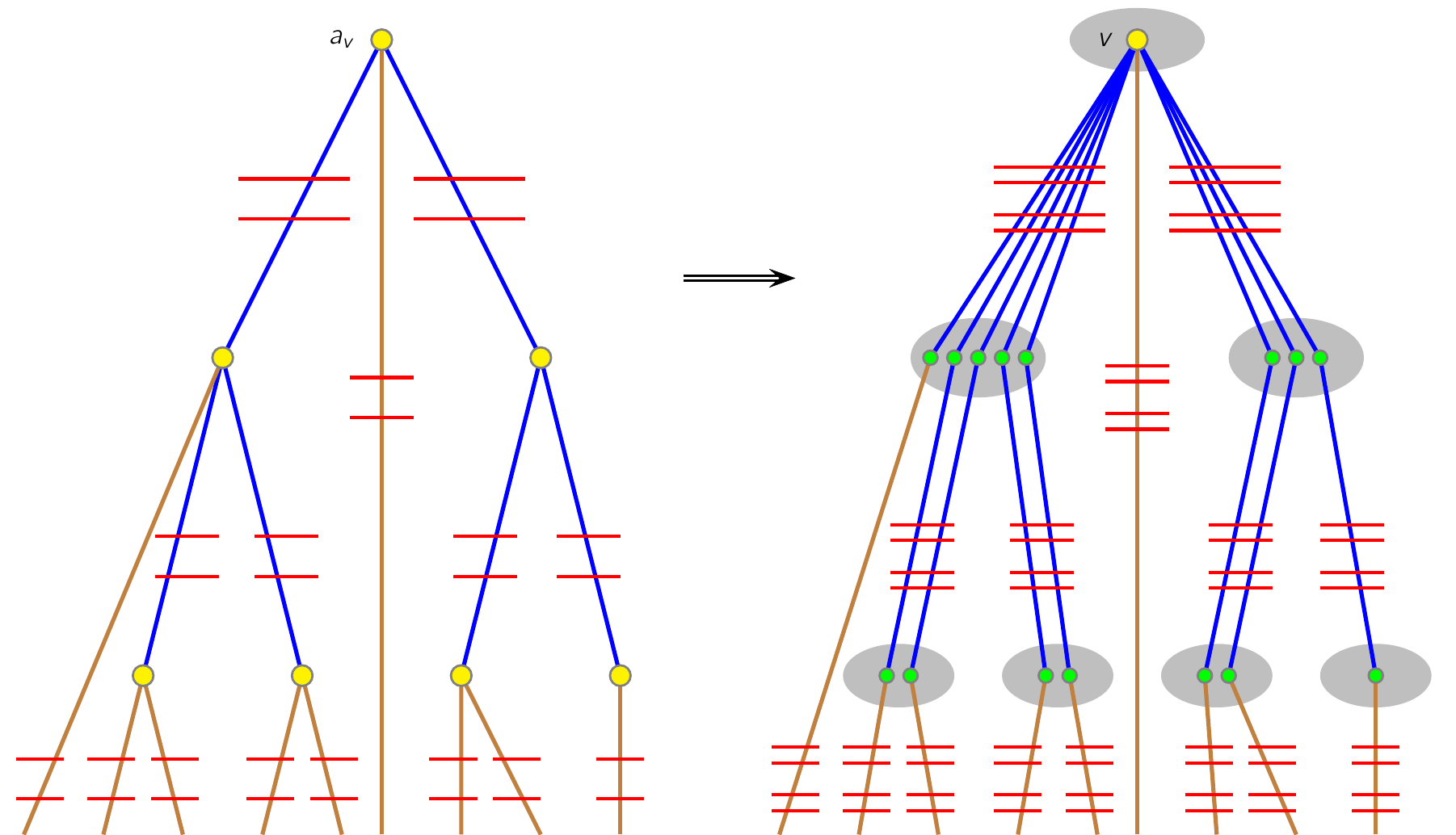}
\vspace*{-3ex}
\caption{Construction of the drawing of $H$.\label{gkCloseH}}
\end{figure}

We now provide the details of this idea. Initialise $V(G'''):= V(G')$ and $E(G'''):=\{ y_{vw}y_{wv}: vw\in E(H)\}$. Consider each vertex $v\in V(H')$. Consider the vertices $y\in V(X'_v)\setminus\{a_v\}$ in non-increasing order of $\text{dist}_{X'_v}(a_v,y)$ (that is, we consider the vertices of $X'_v$ furthest from $a_v$ first, and then move towards the root). Let $x$ be the parent of $y$ in $X'_v$. The incoming edges at $y$ are copies of $xy$. Each outgoing/undirected edge $yz$ at $y$ is already assigned to one edge $vw$ incident to $v$. Say $yz_1,\dots,yz_q$ are the outgoing/undirected edges of $G''$ incident to $y$ in clockwise order in the drawing of $G''$, where $yz_i$ is assigned to edge $vw_i$. 
If $e_1,\dots,e_q$ are the incoming edges at $y$ in clockwise order, then assign $e_{q-i+1}$ to $vw_i$ for each $i\in[q]$. 
Now in $G'''$ replace vertex $y$ by vertices $y_1,\dots,y_q$ drawn in a sufficiently small disc around $y$, where $y_i$ is incident to $e_{q-i+1}$ and $y_iz_i$ in $G'''$. 
Thus the edges in $G'''$ assigned to $vw$ form a path from $a_v$ to $y_{vw}$ and a path from $a_w$ to $y_{wv}$. Hence  $G'''$ is a subdivision of $H'$ (since $y_{vw}y_{wv}$ is an edge of $G'''$). Each edge of $G'''$ has the same number of crossings as the corresponding edge of $G''$. 
Thus, the total number of crossings in the drawing of $G'''$ is at most $2kd^2(2r+1)|E(H')|$. 
Since $G'''$ is a subdivision of $H'$, the drawing of $G'''$ determines a drawing of $H'$ with the same number of crossings. 
Therefore $H$ is $2kd^2(2r+1)$-close to Euler genus $g$. 
\end{proof}

We need the following results of \citet{OOW19}:
\begin{align}
    \label{gkCloseEdges} \rho( \EE_{k,g} ) & \leq  2\sqrt{2k+1}\,\rho_g\\
    \label{K3t} K_{3,3k(2g+3)(2g+2)+2} & \not\in \EE_{g,k}.
\end{align}

We now reach the main result of this section. 

\begin{thm}
\label{gkClose}
For fixed $k,g\in\NN_0$ and every fixed forest $T$, 
$$C(T,\EE_{g,k},n) \in \Theta(n^{\alpha_2(T)}).$$ 
\end{thm}

\begin{proof}
First we prove the lower bound. By \cref{LowerBound} with $s=2$, for all sufficiently large $n\in\mathbb{N}$, there exists a graph $G$ with $|V(G)|\leq n$ and $\tw(G)\leq 2$ and $C(T,G)\geq c_{\ref{LowerBound}}(\alpha_2(T))\, n^{\alpha_2(T)}$. Since $\tw(G)\leq 2$, $G$ is planar and is thus in $\EE_{g,k}$. Hence $C(T,\EE_{g,k},n)\in \Omega(n^{\alpha_2(T)})$. 

Now we prove the upper bound.  Let $s:=2$ and $r:=|V(T)|$ and 
$t:= 54k(2r+1)(2g+3)(2g+2)+2$. 
Let $G$ be an $n$-vertex graph in $\EE_{g,k}$. By \cref{gkCloseEdges}, $\rho(G) \leq 2\sqrt{2k+1}\,\rho_g $. Suppose on the contrary that $I(T,G)\geq cn^{\alpha_2(T)}$ where $c:=c_{\ref{UpperBound}}(s,t,r,2\sqrt{2k+1}\,\rho_g )$. 

Let $H:=K_{3,t}$. \cref{UpperBoundCorollary} implies that $G$ contains a $(1,r)$-model $(X_v:v\in V(H))$ of $H$. This model is $r$-shallow and for every vertex $v\in V(H)$ we have $\deg_H(v)\leq 3$ or $|V(X_v)|=1$. Thus \cref{gkCloseShallow} is applicable with $d=3$, implying that $K_{3,t} \in \EE_{g,18k(2r+1)}$, which contradicts \cref{K3t}.
\end{proof}

An almost identical proof to that of \cref{gkCloseShallow} shows the following analogous result for $\SS_{\Sigma,k}$. This can be used to prove \cref{gkPlanar} without using shortcut systems.

\begin{lem}
\label{gkPlanarShallow}
Fix a surface $\Sigma$ and $k,r\in\NN_0$ and $d\in\NN$. Let $G$ be a graph in $\SS_{\Sigma,k}$ that contains an $r$-shallow $H$-model $(X_v:v\in V(H))$ such that for every vertex $v\in V(H)$ we have $\deg_H(v)\leq d$ or $|V(X_v)|=1$. Then $H$ is in $\SS_{\Sigma,kd^2(2r+1)}$.
\end{lem}

\section{Open Problems}
\label{OpenProblems}

In this paper we determined the asymptotic behaviour of $C(T,\GG,n)$ as $n\to \infty$ for various sparse graph classes $\GG$ and for an arbitrary fixed forest $T$.  One obvious question is what happens when $T$ is not a forest? 

For arbitrary graphs $H$, the answer is no longer given by $\alpha_s(H)$.  \citet{HJW20} define a more general graph parameter, which they conjecture governs the behaviour of $C(H,\GG,n)$.  An \emph{$s$-separation} of $H$ is a pair $(A,B)$ of edge-disjoint subgraphs of $H$ such that $A \cup B=H$, $V(A) \setminus V(B) \neq \emptyset$, $V(B) \setminus V(A) \neq \emptyset$, and $|V(A) \cap V(B)|=s$.  A \emph{$(\leq s)$-separation} is an $s'$-separation for some $s' \leq s$.  Separations $(A,B)$ and $(C,D)$ of  $H$ are \emph{independent} if $E(A) \cap E(C) = \emptyset$ and $(V(A) \setminus V(B)) \cap (V(C) \setminus V(D))=\emptyset$.  If $H$ has no $(\leq s)$-separation, then let $f_s(H):=1$; otherwise, let $f_s(H)$ be the maximum number of pairwise independent $(\leq s)$-separations in $H$.  

\begin{conj}[\citep{HJW20}] \label{flopnumber}
Let $\BB_{s,t}$ be the class of graphs containing no $K_{s,t}$ minor, where $t\geq s \geq 1$. Then for every fixed graph $H$ with no $K_{s,t}$ minor,
 $$C(H,\BB_{s,t},n) \in \Theta(n^{f_{s-1}(H)}).$$ 
\end{conj}

As evidence for \cref{flopnumber}, \citet{Eppstein93} proved it when $f_{s-1}(H)=1$ and \citet{HJW20} proved it when $s \leq 3$ (and that the lower bound holds for all $s \geq 1$). It is easy to show that $f_s(T)=\alpha_s(T)$ for all $s \geq 1$ and every forest $T$. Thus, if true, \cref{flopnumber} would simultaneously generalise \cref{MinorClosedClass} and results from \citep{HJW20}. 

In light of \cref{Degeneracy} we also conjecture the following generalisation. 
\begin{conj}
Let $\DD_k$ be the class of $k$-degenerate graphs. Then for every fixed $k$-degenerate graph $H$,
 $$C(H,\DD_k,n) \in \Theta(n^{f_{k}(H)}).$$ 
\end{conj}

\subsection*{Acknowledgements}

Many thanks to both referees for several helpful comments. 

\subsection*{Note}

Subsequent to this work, \citet{Liu21} disproved Conjectures~16 and 17, amongst many other results.

\def\soft#1{\leavevmode\setbox0=\hbox{h}\dimen7=\ht0\advance \dimen7
  by-1ex\relax\if t#1\relax\rlap{\raise.6\dimen7
  \hbox{\kern.3ex\char'47}}#1\relax\else\if T#1\relax
  \rlap{\raise.5\dimen7\hbox{\kern1.3ex\char'47}}#1\relax \else\if
  d#1\relax\rlap{\raise.5\dimen7\hbox{\kern.9ex \char'47}}#1\relax\else\if
  D#1\relax\rlap{\raise.5\dimen7 \hbox{\kern1.4ex\char'47}}#1\relax\else\if
  l#1\relax \rlap{\raise.5\dimen7\hbox{\kern.4ex\char'47}}#1\relax \else\if
  L#1\relax\rlap{\raise.5\dimen7\hbox{\kern.7ex
  \char'47}}#1\relax\else\message{accent \string\soft \space #1 not
  defined!}#1\relax\fi\fi\fi\fi\fi\fi}

\end{document}